\documentclass[11pt]{article}

\usepackage[margin=1in]{geometry}
\usepackage{amsmath,amssymb,amsthm,mathtools}
\usepackage{enumitem}
\usepackage{booktabs}
\usepackage{nicefrac}
\usepackage{array}
\usepackage{microtype}
\usepackage{xcolor}
\usepackage{graphicx}
\usepackage{float}
\usepackage{tikz}
\usetikzlibrary{arrows.meta,positioning,fit}
\usepackage{xurl}
\usepackage{hyperref}
\usepackage[nameinlink,capitalize,noabbrev]{cleveref}
\usepackage{authblk}

\allowdisplaybreaks
\setlist[itemize]{leftmargin=2em,itemsep=0.2em,topsep=0.4em}
\setlist[enumerate]{leftmargin=2.2em,itemsep=0.2em,topsep=0.4em}

\newtheorem{theorem}{Theorem}[section]
\newtheorem{lemma}[theorem]{Lemma}
\newtheorem{proposition}[theorem]{Proposition}
\newtheorem{corollary}[theorem]{Corollary}

\theoremstyle{definition}
\newtheorem{definition}[theorem]{Definition}
\newtheorem{example}[theorem]{Example}
\newtheorem{question}[theorem]{Question}
\theoremstyle{remark}
\newtheorem{remark}[theorem]{Remark}

\newcommand{\E}{\mathbb{E}}
\newcommand{\Prb}{\mathbb{P}}
\newcommand{\R}{\mathbb{R}}
\newcommand{\C}{\mathbb{C}}
\newcommand{\Z}{\mathbb{Z}}
\newcommand{\one}{\mathbf{1}}
\newcommand{\Gammahat}{\widehat{\Gamma}}
\newcommand{\Cay}{\operatorname{Cay}}
\newcommand{\ord}{\operatorname{ord}}

\newcommand{\roundval}{\operatorname{Round}}

\newcommand{\Ree}{\operatorname{Re}}
\newcommand{\eps}{\varepsilon}

\newcommand{\gl}{\operatorname{SDP_{GL}}}

\title{\textbf{Integrality Gap Bounds for the Goemans–Linial SDP on Finite Abelian Cayley Graphs}}

\author{Georgios Stamoulis\thanks{georgios.stamoulis@maastrichtuniversity.nl}}
\affil{Department of Advanced Computing Sciences \\ Maastricht University \\ The Netherlands}
\date{}

\begin{document}
\maketitle

\begin{abstract}
In the uniform sparsest cut problem we are asked to find a vertex set that cuts few edges relative to the number of vertex pairs it separates.  The  Goemans-Linial semidefinite relaxation coupled with the Arora-Rao-Vazirani rounding procedure gives an $\mathcal{O}(\sqrt{\log n})$ approximation ratio on arbitrary graphs on $n$ vertices. We study this relaxation on finite Abelian Cayley graphs and we prove the following results: Firstly we show that when the second normalized Laplacian eigenvalue of $G= \Cay(\Gamma, S)$ is realized by a \emph{Fourier character} the image of which has size \emph{at most four} then $\lambda_2(G) ~=~ \gl (G) ~=~ \psi(G)$, where $\gl(G)$ is the semidefinite optimum and $\psi(G)$ is the uniform sparsest cut optimum.  The reason is geometric: a character maps the vertices onto a regular polygon where the squared chord distance on that polygon satisfies the triangle inequalities exactly when the polygon has at most four vertices.  Grouping vertices with the same character value we produce a small \emph{cyclic quotient} in which an optimal cut can be then found exactly.  As a corollary we get that the relaxation is exact on finite Abelian Cayley graphs over a group of exponent at most four.

Secondly, for arbitrary orders we replace each generator $s$ of $S$  by a uniformly random element of its entire cyclic subgroup (including the identity) of size $r_s=\operatorname{order}(s)$ and  define by $\alpha(r)$ to be the average number of original $+s$ or $-s$ steps that we need in order to simulate a uniformly random move inside a cyclic subgroup of order $r$. Then, we define $\rho(S) = \max ~\alpha(r_s)$ to be its worst case (among all $s$ in $S$). We show that a full cyclic averaging eliminates the phase of each character and by then choosing a nontrivial character $\chi^*$ that minimizes the auxiliary eigenvalue and taking the corresponding kernel cut $\ker \chi^*$ we get a cut which satisfies
\[
    \psi(G)
      ~\leq~ \psi_G(\ker \chi^*)
      ~\leq~ \frac{q^*}{q^*-1} \rho(S) \cdot \gl (G)
      ~\leq~ 2\rho(S) \cdot \gl(G)
\]
where $q^*$ is the number of values taken by the selected character. 

Finally, we also give an explicit infinite family of Abelian Cayley graphs with GL gap exactly $\nicefrac{16}{15}$.  
\end{abstract}


\section{Introduction}\label{sec:introduction}


A cut is useful when it separates many pairs of vertices while at the same time  cutting few graph edges. This is well-known as the  \emph{uniform sparsest cut} problem. For a regular graph $G=(V,E)$ and a nonempty set $A\subsetneq V$, we consider the following two experiments:
\begin{enumerate}
  \item we choose uniformly a random directed edge and we ask whether it crosses $A$,
  \item we choose uniformly two independent random vertices and we ask whether they fall on opposite sides of $A$.
\end{enumerate}
The ratio of the first probability to the second is the \emph{uniform sparsity} of $A$. The optimum, which we denote by $\psi(G)$, is the minimum of this ratio over all (nontrivial) cuts $A$.

The uniform sparsest cut problem is a very well studied problem in the approximation algorithm literature. Leighton and Rao in  \cite{LR99} used LPs to obtain an $\mathcal{O}(\log n)$ approximation ratio. Arora, Rao and Vazirani in their seminal work \cite{ARV09} improved the approximation factor to $\mathcal{O}(\sqrt{\log n})$ by using a semidefinite relaxation and a beautiful geometric rounding schema. The SDP relaxation itself is commonly known as the Goemans-Linial SDP and the rounding of that SDP as the ARV rounding.  

The metric formulation of the problem is the following. A cut $A$ determines a natural $0$-$1$ distance
\[
  d_A(x,y)=|\one_{A}(x)-\one_{A}(y)|
\]
where $\one_A(z) = 1$ if $z \in A$ and zero otherwise. The cut objective is the ratio between the average value of $d_A$ on graph edges and its average value on all vertex pairs. In order to obtain a relaxation, Goemans and Linial first considered the $n$-dimensional relaxation $v_x$ of each vertex $x$, and then replaced $d_A$ by a much larger class of distances, the squared Euclidean distances
\[
  d(x,y)=\|v_x-v_y\|_2^2
\]
that also satisfy the triangle inequalities. We will use $\gl(G)$ for the optimum relaxed value of the Goemans-Linial SDP on  $G$ and $\mathrm{ARV}(G)$ for the value obtained after the ARV rounding is applied to that SDP. Each cut metric is a feasible solution for the Goemans-Linial SDP which gives $\gl(G) \leq \psi(G)$. On the other hand, let us consider a feasible squared Euclidean metric and  write $v_x = (f_1(x), \dots, f_m(x))$. Then, the square Euclidean distances decompose as
\[
||v_x - v_y||_2^2 = \sum_{j=1}^m |f_j(x) - f_j(y)|^2.
\]
The variational, or Rayleigh, characterization of $\lambda_2(G)$ of the Laplacian says that no matter which scalar function $f_j$ we choose, the average squared variation across edges is at least $\lambda_2(G)$ times the average squared variation across all pairs of vertices. If we sum this over the coordinates gives us that the same lower bound holds for the vector valued metric. As such, each feasible Goemans-Linial SDP solution must have value at least $\lambda_2(G)$ giving us the fundamental inequalities 
\begin{equation}\label{eq:fundamental-chain-intro}
  \lambda_2(G) ~\leq~ \gl(G) ~\leq~ \psi(G).
\end{equation}

The ARV theorem says that $\psi(G)/\gl(G)=\mathcal{O}(\sqrt{\log n})$ for every $n$-vertex graph with uniform demands. The  worst-case gap in the uniform-demand case remains open: the best known lower bounds are much smaller, with Kane and Meka proving a bound of $\exp\{\Omega(\sqrt{\log\log n})\}$ \cite{KM13}. For the more general problem with arbitrary capacities and demands, the Goemans-Linial gap is now known to be $\Theta(\sqrt{\log n})$ since Naor and Young proved the lower bound \cite{NY17}, and Chang, Naor, and Ren proved the matching upper bound \cite{CNR25}. In this manuscript we consider only uniform demands for which the denominator is the uniform average over all vertex pairs. 

\subsection{Abelian Cayley graphs and a summary of our results}

A Cayley graph $G=\Cay(\Gamma,S)$ is built from a group $\Gamma$ and a multiset $S$ of allowed steps or moves. Its vertices are the group elements, and from $x$ we may move to $x+s$ for $s\in S$. When $\Gamma$ is Abelian, its one-dimensional Fourier characters diagonalize every Cayley random walk and each eigenvector has explicit algebraic form. This class of graphs contains cycles, Cartesian products of cycles, hypercubes, cubelike graphs, and Cayley graphs over vector spaces such as $\mathbb{F}_p^k$. These graphs can have very poor expansion, large eigenspaces, or large degree.
 
In \cite{Trevisan21}, Oveis Gharan and Trevisan showed that for a $d$-regular Cayley Graph $G$ of an Abelian group we have that $\psi(G) ~\leq~ \mathcal O (\sqrt{d}) \cdot\gl (G)$, and it was conjectured that sparsest cut on Cayley graphs may admit a constant factor approximation ratio. One step towards the resolution of this conjecture is the recent work of d'Orsi et al. \cite{dorsiaetall25} in which they provide an approximation \emph{schema} the run time of which depends exponentially on the degree $d$ of $G$.  The cyclic-averaging construction developed here was directly inspired by their scalar-multiple auxiliary graph over $\mathbb{F}_p^k$. We describe the precise relationship and the extensions obtained here in subsection \ref{subsect:relatedwork}.

Now we give a high level description of our results. First some terminology: Let $d$ be a semimetric on $\Gamma$, and let $\chi:\Gamma\to\C$ with $|\chi(x)|=1$ be a non-trivial \emph{Fourier character}. Let us write
\[
N_G(d)=\mathbb{E}_{x\sim\Gamma,\,s\sim S}d(x,x+s),
\qquad
D(d)=\mathbb{E}_{x,y\sim\Gamma}d(x,y),
\]
for its average value on a random directed graph step and on a pair of
independent uniform vertices, respectively. Also, let 
\[
  d_\chi(x,y)=|\chi(x)-\chi(y)|^2.
\]
be the corresponding squared Euclidean distance: it is the squared chord length between two points on the unit circle. The normalized objective ratio is exactly the character eigenvalue:
\[
  \frac{N_G(d_\chi)}{D(d_\chi)}=\lambda_\chi(G).
\]
Note that a character attaining $\lambda_2$ does not necessarily imply that $\gl(G)=\lambda_2(G)$ because squared Euclidean distances do not necessarily satisfy the triangle inequalities i.e., they are not necessarily valid GL (Goemans-Linial) solutions. A character with image of size $q$ traces the regular $q$-gon and its squared chordal distance is a metric only when $q \leq 4$. This leads to the first result of the paper. Low order character images are  feasible GL solutions and they can also be rounded without loss by cutting the cyclic quotient seen by the character. The exactness theorem we prove is actually slightly stronger than a low-exponent statement and can be described by the following:

\begin{theorem}
\label{thm:intro-exact}
Let $G=\Cay(\Gamma,S)$ and suppose that some character $\chi$ realizing $\lambda_2(G)$ has image of size $q \leq 4$. Then 
\[
    \lambda_2(G)=\gl(G)=\psi(G).
\]
If $q \in \{2,3\}$ then the kernel of $\chi$ (the set of group elements that map to the identity element under that character) is an optimal cut. If $q=4$ then one of two half-circle cuts in the cyclic quotient is optimal.
\end{theorem}

Since each character image order divides the exponent of $\Gamma$ we get that

\begin{corollary}\label{cor:intro-exp4}
If $\exp(\Gamma)\le4$ then every Cayley graph $G=\Cay(\Gamma,S)$ satisfies $\lambda_2(G)=\gl(G)=\psi(G)$.
\end{corollary}

The geometric idea is the following exact polygon threshold 
\[
  |\omega^a- \omega^b|^2\text{ satisfies all triangle inequalities on }\Z_q
  \quad\Longleftrightarrow\quad q\le4.
\]
This is the regular-polygon instance of the classical nonobtuse-set criterion that squared Euclidean distance is a metric on a finite point set precisely when every triangle spanned by the set is nonobtuse.

For arbitrary orders we let $r_s=\ord(s)$ and  define by $\alpha(r)$ to be the average number of original $+s$ or $-s$ steps that we need so that we are able to simulate a uniformly random move inside a cyclic subgroup of order $r$. Then, we define $\rho(S)$ to be its worst case (among all $s$ in $S$).

In the cyclically averaged random walk we first choose a move $s\in S$ and then choose a uniformly random  multiple $\ell s\in\langle s\rangle$  (including the identity element) where $\ell$ itself is also uniform in $\{0,1, \dots, r_s-1\}$. We denote this auxiliary graph by $G^\sharp$. The quantity $\rho(S)$ is the  price we have to pay when we transfer the GL objective from the original graph to the auxiliary walk.

\begin{theorem}
\label{thm:intro-general}
For each metric $d$ on $\Gamma$ we have $N_{G^\sharp}(d) ~\leq~ \rho(S) \cdot N_G(d)$. For each character $\chi$ 
\[
  \lambda_\chi(G^\sharp)
  ~=~ \E_{s\sim S}
     \one_{\{\chi(s)\neq 1\}} ~=~ \Pr_{s \sim S}~[\chi(s) \neq 1] 
\]
Let $\chi^*$ be a character that minimizes this expression among all the nontrivial characters and set $q^*=|\chi^*(\Gamma)|$, where $\chi^*(\Gamma) = \{\chi^*(x): x \in \Gamma \}$. Then
\[
  \psi(G)
  ~\leq~ \roundval(\chi^*)
  ~\leq~ \frac{q^*}{q^*-1}\rho(S)\gl(G)
  ~\leq~ 2\rho(S)\gl(G).
\]
\end{theorem}

Here $\roundval(\chi)$ is the optimum cut value in the \emph{cyclic quotient induced by $\chi$}, lifted (or pulled) back to $G$. The theorem above is constructive. In the explicit input model in which the decomposition of $\Gamma$ and the generator multiset are given and $n=|\Gamma|$ is the size parameter, we can simply enumerate all the $n$ characters, then minimize the corresponding formula, and finally output $\ker\chi^*$. This gives a polynomial time $2\rho(S)$-approximation algorithm without ever solving the SDP.

We also compute the exact loss in the metric comparison: we argue that $\alpha(r)$ is optimal even over cut metrics on $\Z_r$. We also show that the image size factor of $q/(q-1)$ is optimal for the  \emph{kernel rounding}. Of course, these two facts by themselves do not necessarily  imply that the integrality gap bound from the above composition is tight because an optimal quotient cut could very well be much better than the kernel and also because the intermediate inequalities may not be tight at the same time.

Cycles and prime-exponent vector spaces illustrate these two cases. First, we consider cycles that satisfy
\[
  \gl(C_n)=\psi(C_n)=
  \begin{cases}
    4/n,&n\text{ even},\\[2pt]
    4n/(n^2-1),&n\text{ odd},
  \end{cases}
\]
but on the other hand $\rho(\{+1/-1\}) = \Theta(n)$. 
Second, for each connected Cayley graph over $\mathbb{F}_p^k$, $p$ prime, we have that 
\[
  \frac{\psi(G)}{\gl(G)}
  ~\leq~ \frac{p}{p-1} \cdot \alpha(p)
  ~=~\frac{p +1}{4},
\]
independently of the degree. 

We note that as an existential integrality-gap bound this dependence on $p$ is not optimal as 
Proposition 3.3 of Austin, Naor, and Valette \cite{DBLP:journals/dcg/AustinNV10} gives an $\mathcal O (\log p)$ bound. On the other hand, the estimate above gives explicit coefficient and a direct character-quotient rounding, which does not require computing an SDP solution or an $L_1$ embedding.


As our final result we show that the Goemans–Linial relaxation is not exact on Abelian Cayley graphs by providing a small 10-vertex graph with integrality gap 16/15. Its Cartesian powers give an infinite connected family with the same gap. For the amplification mechanism we use exact tensorization of the cut optimum  results of Bonsma \cite{bonsmatensor}.

\subsection{Related work}
\label{subsect:relatedwork}
Several ingredients in our paper are well known in the literature. It is standard that Fourier characters diagonalize Abelian Cayley graphs \cite{SteinbergGroups, kantor2015mathematics, Trevisan21}, that squared Hilbert distances are the negative-type objects of Schoenberg \cite{Schoenberg38}, and that triangle inequality comparisons on auxiliary Cayley graphs already appear in the degree argument of Oveis Gharan and Trevisan \cite{Trevisan21}.

The metric approach for the sparsest cut problem was initially proposed by the work of Leighton \& Rao \cite{LR99}, and was followed by the metric-embedding of Linial-London-Rabinovich \cite{LLR95}, and the semidefinite/expander-flow work of Arora-Rao-Vazirani \cite{ARV09}. 
The general and uniform-demand lower bounds discussed above shows that additional graph structure is needed to improve the generic guarantees \cite{KM13,NY17,CNR25}. 

Metric embeddings associated with Abelian Cayley graphs have also been studied from a complementary direction. In  \cite{DBLP:journals/toc/NewmanR09} the authors constructed Cayley graphs on finite Abelian groups whose shortest path metrics require $\Omega (\log |\Gamma|)$ distortion when embedded into Euclidean or negative-type metrics. 
An argument due to Oveis Gharan and Trevisan, presented in detail by Trevisan, shows that a degree-$d$ Abelian Cayley graph has GL integrality gap $\mathcal O(\sqrt d)$ \cite{Trevisan21}. In their proof they use a walk-power auxiliary graph. Our averaging comparison argument borrows the same underlying idea of using an auxiliary graph, but since the parameter of our interest is the \emph{generator} order, and not the degree, we treat each cyclic generator direction separately.


Recent work of d'Orsi, Jones, Ruotolo, Vadhan, and Zhang gives a $(1+\eps)$-approximation schema for degree-$d$ Abelian Cayley graphs running in time $n^{O(1)}\exp\{(d/\eps)^{O(d)}\}$. The approach they use works through eigenspace enumeration and cut improvement \cite{dorsiaetall25}. In the special case where $\Gamma=\mathbb{F}_p^k$ in their Section 8 they form an auxiliary graph by taking all the \emph{nonzero} scalar multiples of every generator and they use it to obtain an $\mathcal{O}(p)$ approximation ratio for sparsest cut. They also note that this approximation guarantee also follows by \cite{DBLP:conf/stoc/KwokLLGT13} but they provide a simpler algorithm and analysis. That auxiliary graph is essentially our cyclic averaging $G^\sharp$ where we additionally include the zero element. The inclusion of the zero element makes each subgroup average an exact Fourier projection and allows us to extend the construction beyond prime-exponent vector spaces by allowing arbitrary generator  orders. 
Note that if $G'$ denotes their non zero-multiple walk and $G^\sharp$ ours then
\[
A_{G^{\sharp}} = \frac{1}{p} \mathbb{I} + \frac{p-1}{p}A_{G'} ~\mbox{ and }~ L_{G^\sharp} = \frac{p-1}{p} L_{G'}
\]
This addition of the zero element removes the factor of $\nicefrac{p}{p-1}$  from the spectrum of the character. Our argument in Section \ref{sec:averaging} also follows the character kernel rounding structure they use within their Section 8.  We also note that in our work, unlike their approximation guarantee, our comparison holds for all feasible metric with respect to the GL SDP and gives a direct bound on the GL integrality gap. D’Orsi et al. do not analyze the GL SDP optimum and their approximation guarantee does not imply an integrality gap bound.

Finally, we  mention that the geometric arguments underlying the polygon threshold are standard, see \cite{DG62} and \cite{DL97}. Our  usege  is to identify when an appropriate Fourier character is a feasible solution of the Goemans-Linial  relaxation and how it can be rounded without loss.


\section{Preliminaries}\label{sec:preliminaries}

In this section we present some basic definitions and facts used later. These are known, but we include them here for completeness and we give pointers to some excellent relevant references. 

\begin{table}[H]
\centering
\small
\renewcommand{\arraystretch}{1.18}
\begin{tabular}{>{\raggedright\arraybackslash}p{0.19\textwidth} >{\raggedright\arraybackslash}p{0.69\textwidth}}
\toprule
\textbf{Symbol} & \textbf{Explanation} \\
\midrule
$\Gamma$, $\Gammahat$ & A finite Abelian group and its group of complex characters. \\
$S$, $G=\Cay(\Gamma,S)$ & A symmetric generator multiset and its normalized Cayley random walk. \\
$\lambda_2(G)$ & The second normalized Laplacian eigenvalue. \\
$N_G(d)$, $D(d)$ & The average metric value on graph steps and on independent uniform vertex pairs. \\
$\psi(G)$, $\gl(G)$ & The optimum uniform sparsest cut value and the optimum Goemans-Linial metric relaxation value. \\
$q(\chi)$, $Q_\chi$ & The size of a character image and the induced cyclic quotient graph. \\
$G^\sharp$ & The cyclically averaged graph . \\
\bottomrule
\end{tabular}
\caption{Notation used throughout the paper. All of them are formally defined below.}
\label{tab:notation}
\end{table}

\subsection{Probability and multigraph conventions}


A \emph{symmetric} multiset $S$ over some group $\Gamma$ contains $s$ and $-s$ with the same multiplicity. We always assume that $S$ is nonempty and that $0 \notin S$ for the original input graph. Quotient graphs may contain projected zero steps, which are self-loops and we include those without any loss because they are part of the random walk normalization.

We follow the common convention that a ``metric'' may assign distance zero to distinct points. Formally, such an object is a \emph{semimetric}: it is nonnegative, symmetric, vanishes on the diagonal, and satisfies the triangle inequalities, but does not necessarily separate points. This convention includes cut metrics.


For $g\in\Gamma$ we denote its \emph{order} by $\ord(g)=\min\{r \geq 1:rg=0\}$ and its \emph{exponent}  by $\exp(\Gamma)=\operatorname{lcm}\{\ord(g):g\in\Gamma\}$. Equivalently, it is the least positive integer $m$ such that $mg=0$ for each $g\in\Gamma$. The standard structure theorem \cite{dummitfoote} gives an isomorphism
\begin{equation}\label{eq:invariant-factors}
  \Gamma\cong \Z_{n_1}\times\cdots\times\Z_{n_k},
  \qquad n_1\mid n_2\mid\cdots\mid n_k.
\end{equation}
In this representation, $\exp(\Gamma)=n_k$.

Uniform measure on $\Gamma$ is translation invariant \cite{analysistopologyhowes}: if we fix an $h\in\Gamma$ and a function $f$ from $\Gamma$ to $\C$ then 
\begin{equation}\label{eq:translation-invariance}
  \E_{x\sim\Gamma}f(x+h)=\E_{x\sim\Gamma}f(x)
\end{equation}
and this elementary fact is the reason that we can average path lengths in a nice way during our the averaging argument.

\subsection{Characters and Fourier analysis}
The books \cite{kantor2015mathematics, Terras99} provide a thorough treatment on Fourier analysis.  A character is a homomorphism $\chi:\Gamma\to\{z\in\C:|z|=1\}$. The set of characters is the dual group $\Gammahat$. The trivial character, equal to $1$ everywhere, is denoted $\mathbf{1}$.

\begin{proposition}[Some basic character facts]\label{prop:character-facts}
For each $\chi\in\Gammahat$ and $g\in\Gamma$:
\begin{enumerate}[label=(\roman*)]
  \item $\chi(0)=1$ and $\chi(-g)=\overline{\chi(g)}$;
  \item if $\ord(g)=r$, then $\chi(g)^r=1$;
  \item $\ker\chi=\{g:\chi(g)=1\}$ is a subgroup;
  \item if $\chi\ne\mathbf{1}$, then $\E_x\chi(x)=0$.
\end{enumerate}
\end{proposition}

The characters form an orthonormal basis of $L^2(\Gamma)$ under the inner product
$\langle f,g\rangle=\E_{x\sim\Gamma}f(x)\overline{g(x)}$. For completeness we write $\Gamma$ as in \eqref{eq:invariant-factors}. For $a=(a_1,\ldots,a_k)$ with $a_j\in\Z_{n_j}$ we define
\[
  \chi_a(x_1,\ldots,x_k)
  ~=~ \prod_{j=1}^k \exp \left(\frac{2\pi i a_jx_j}{n_j}\right).
\]
There are exactly $|\Gamma|$ such  orthogonal functions.

The image of a character is a finite subgroup of the unit circle and is therefore cyclic. Indeed, if $m$ is the least common multiple of the orders of its elements, then the image lies in the cyclic group of $m$th roots of unity, and every subgroup of a cyclic group is cyclic. We write $q(\chi)=|\chi(\Gamma)|.$
By the first isomorphism theorem,
\begin{equation}\label{eq:kernel-density}
  |\Gamma|=|\ker\chi|\,q(\chi),
  \qquad
  \mu(\ker\chi)=\frac1{q(\chi)}.
\end{equation}
Moreover, $q(\chi)$ divides $\exp(\Gamma)$. Conversely, if $\Gamma$ has exponent $m$ then the projection into the last cyclic factor in \eqref{eq:invariant-factors}, followed by the standard character of $\Z_m$, gives a character with image size $m$.

\subsection{Abelian Cayley graphs and their spectra}

Let $S$ be a finite symmetric multiset in $\Gamma\setminus\{0\}$. The Cayley graph $ G=\Cay(\Gamma,S)$ is the regular multigraph whose normalized random walk moves from $x$ to $x+s$ after sampling an $s\sim S$. Its degree is $|S|$, counted with multiplicity.  The normalized adjacency operator and normalized Laplacian are respectively
\[
  (A_Gf)(x)=\E_{s\sim S}f(x+s),
  \qquad
  L_G=I-A_G.
\]

\begin{proposition}[Character diagonalization]\label{prop:character-diagonalization}
For each character $\chi$ $A_G\chi=a_\chi(G)\chi$ and $a_\chi(G)=\E_{s\sim S}\chi(s)\in\R$.
The corresponding normalized Laplacian eigenvalue is
\begin{equation}\label{eq:character-eigenvalue}
  \lambda_\chi(G)
  =1-a_\chi(G)
  =1-\E_{s\sim S}\Ree\chi(s).
\end{equation}
Consequently,
\begin{equation}\label{eq:lambda2-character-min}
  \lambda_2(G)=\min_{\chi\ne\mathbf{1}}\lambda_\chi(G),
\end{equation}
where eigenvalues are counted with multiplicity, so $\lambda_2(G)=0$ when $G$ is disconnected.
\end{proposition}

For a real valued function $f$ we let $\bar f=\E_xf(x)$. The usual Rayleigh characterization takes the normalization
\begin{equation}\label{eq:rayleigh-normalization}
  \lambda_2(G)
  =\inf_{f\neq \mathrm{const}}
  \frac{\E_{x,s}(f(x)-f(x+s))^2}
       {\E_{x,y}(f(x)-f(y))^2}.
\end{equation}

\subsection{Uniform sparsest cut}
For a nonempty proper set $A\subsetneq\Gamma$ we define the corresponding edge-crossing probability as
\begin{equation}\label{eq:cut-numerator}
  N_G(A)=\Prb_{x\sim\Gamma,\,s\sim S}
  [\one_A(x)\ne\one_A(x+s)].
\end{equation}
This is the fraction of oriented random-walk transitions that cross the cut. The uniform all-pairs separation probability is
\begin{equation}\label{eq:cut-denominator}
  D(A)=\Prb_{x,y\sim\Gamma}[\one_A(x)\ne\one_A(y)]
  =2\mu(A)(1-\mu(A)).
\end{equation}
The \emph{uniform sparsity} of $A$ and the optimum are
\begin{equation}\label{eq:psi-def}
  \psi_G(A)=\frac{N_G(A)}{D(A)},
  \qquad
  \psi(G)=\min_{\emptyset\ne A\subsetneq\Gamma}\psi_G(A).
\end{equation}
The cut metric associated with $A$ is $d_A(x,y)=|\one_A(x)-\one_A(y)|$. Then
\begin{equation}\label{eq:cut-metric-ratio}
  \psi_G(A)
  =\frac{\E_{x,s}d_A(x,x+s)}{\E_{x,y}d_A(x,y)}.
\end{equation}

\subsection{The Goemans-Linial relaxation}
We call a function $d:\Gamma\times\Gamma\to\R_{\geq 0}$ squared Euclidean semimetric if we can find vectors $v_x$  such that $d(x,y)=\|v_x-v_y\|_2^2$.
We call such a $d$ \emph{feasible} if on top of the previous it also satisfies all triangle inequalities $d(x,z)\leq d(x,y)+d(y,z)$ for all $x,y,z\in\Gamma$.

If $X$ is the Gram matrix $X_{xy}=\langle v_x,v_y\rangle$ then $d(x,y)=X_{xx}+X_{yy}-2X_{xy}$.
The condition $X\succeq0$ is semidefinite and the normalization, the triangle inequalities, and the objective are all linear in $X$ thus the optimization below is a polynomial-size semidefinite program.

\begin{lemma}\label{lem:cut-feasible}
For any nontrivial $A\subsetneq\Gamma$ the cut metric $d_A$ is a feasible solution for $\gl$.
\end{lemma}

If for a nonzero feasible metric we define
\begin{equation}\label{eq:metric-ND}
  N_G(d)=\E_{x\sim\Gamma,\,s\sim S}d(x,x+s),
  \qquad
  D(d)=\E_{x,y\sim\Gamma}d(x,y)
\end{equation}
then the corresponding SDP value is
\begin{equation}\label{eq:arv-def}
  \gl(G)=\inf_{\substack{d\text{ is feasible}\\D(d)>0}}
  \frac{N_G(d)}{D(d)}.
\end{equation}

Finally, we state the fundamental inequality connecting $\lambda_2$, the GL value and the true optimum on a given graph instance $G$:
\begin{equation}\label{eq:fundamental-chain}
  \lambda_2(G) ~\leq~ \gl(G) ~\leq~ \psi(G).
\end{equation}
The basic SDP integrality gap is $\psi(G)/\gl(G)$.

\subsection{Character metrics}
For a character $\chi$ let us define $d_\chi(x,y)=|\chi(x)-\chi(y)|^2$ which is squared Euclidean via the mapping $x\mapsto(\Ree\chi(x),\operatorname{Im}\chi(x))\in\R^2$. Its objective has an exact spectral interpretation.

\begin{lemma}\label{lem:character-metric-ratio}
For every nontrivial character $\chi$, we have that $N_G(d_\chi)=2\lambda_\chi(G)$, $D(d_\chi)=2$ and therefore $\frac{N_G(d_\chi)}{D(d_\chi)} ~=~ \lambda_\chi(G)$.
\end{lemma}

\begin{proof}
Because $|\chi(x)|=1$ and $\chi(x+s)=\chi(x)\chi(s)$,
\[
  N_G(d_\chi)
  ~=~ \E_s|1-\chi(s)|^2
  ~=~ 2 \bigl(1-\E_s\Ree\chi(s)\bigr)
  ~=~ 2 \lambda_\chi(G).
\]
For the denominator, if $x,y$ are independent uniform points then $u=y-x$ is uniform and so using the Proposition \ref{prop:character-facts} $D(d_\chi) ~=~ \E_u|1-\chi(u)|^2
  ~=~ 2 \bigl(1-\Ree\E_u\chi(u)\bigr)=2$
\end{proof}

The character metric  certifies $\gl(G)=\lambda_2(G)$ if a minimizing character metric satisfies the triangle inequalities. This is precisely the role of the next section.

\section{Quotient rounding}
\label{sec:quotient}


The fibers of a character $\chi$ are the cosets of $\ker \chi$. Since $\chi(\Gamma)$ is cyclic, we can collapse these fibers and get a Cayley graph on $\mathbb{Z}_q$ where $q=|\chi(\Gamma)|$. A cut on this quotient graph can be lifted cuts that correspond to unions of fibers in the original graph preserving sparsity.

\begin{definition}[Character quotient]\label{def:character-quotient}
Let $\omega = e^{2\pi i/q}$. $\mathbb{Z}_q$ maps to $\chi(\Gamma)$ with $a \mapsto \omega^a$. We define $a_s \in \mathbb{Z}_q$ as $\chi(s) = \omega^a$. The quotient Cayley multigraph associated with $\chi$ is then
\[
  Q_\chi ~=~ \Cay(\Z_q,\{a_s:s\in S\}),
\]
with the same multiplicities as $S$. If $a_s=0$ then the corresponding quotient step corresponds to a self loop and still remains in the random walk distribution.
\end{definition}

For $B\subseteq\Z_q$ we identify $B$ with the corresponding subset of $\chi(\Gamma)$ and define its lift, or pullback: 
\[
  \chi^{-1}(B) ~=~ \big\{x\in\Gamma:~\chi(x)\in B \big\}.
\]

The following elementary lemma shows that quotient rounding incurs no loss i.e., each cut of the cyclic quotient corresponds  back to a cut on the original graph with exactly the same sparsity:

\begin{lemma}[Exact pullback identity]\label{lem:pullback-identity}
For every nonempty proper $B\subsetneq\Z_q$, $\psi_G(\chi^{-1}(B))=\psi_{Q_\chi}(B)$.
\end{lemma}

\begin{proof}
Every fiber of the surjective homomorphism $\chi:\Gamma\to\chi(\Gamma)$ has size $|\ker\chi|$. Thus $\chi(x)$ is uniform in $\Z_q$ when $x$ is uniform in $\Gamma$, and $\mu_\Gamma(\chi^{-1}(B))=\frac{|B|}{q}$. Additionally
\[
  \chi(x+s)=\chi(x)\chi(s)
  \quad\longleftrightarrow\quad
  u\mapsto u+a_s.
\]
Therefore a random step $(x,x+s)$ crosses the pullback cut exactly when the corresponding quotient step $(u,u+a_s)$ crosses $B$. The numerators and denominators in \eqref{eq:psi-def} remain the same.
\end{proof}

\begin{definition}
\label{def:round-value}
For a nontrivial character $\chi$  the value of the Quotient rounding is $\roundval(\chi) = \psi(Q_\chi)
  = \min_{\emptyset\ne B\subsetneq\Z_q}\psi_{Q_\chi}(B)$.
\end{definition}

A quotient cut corresponds through the pull back into an actual cut in $G$ and that the singleton $\{0\}\subset\Z_q$ similarly corresponds to the $\ker\chi$. This implies the following corollary 

\begin{corollary}
\label{cor:quotient-kernel-chain}
For every nontrivial $\chi$, $\psi(G) \leq \roundval(\chi) \leq \psi_G(\ker\chi)$.
\end{corollary}

\section{Exactness from low-order characters}\label{sec:exact}
In this section we show how to combine the lossless lifting principle of the kernel rounding (that simply chooses the single quotient vertex corresponding to the identity the lift of which is $\ker\chi$) with the geometry of the  regular polygon induced by the characters on the unit cycle. For character images of
size at most four  the squared-chord metric is feasible for GL and we show how an appropriate cut of the cyclic quotient has no larger objective value. This implies that the spectral, semidefinite, and cut optima to coincide.

\subsection{Squared chordal distance on regular polygons}

Let $\omega=e^{2\pi i/q}$ where $q = q(\chi)$ as before. The character image of order $q$ carries the \emph{squared chordal semimetric} $d_q(a,b)=|\omega^a-\omega^b|^2$ where $a,b\in\Z_q$.

\begin{lemma}
\label{lem:polygon-cutoff}
For $q \geq 2$ the squared chordal semi-metric $d_q$ satisfies all triangle inequalities if and only if $q\leq4$.
\end{lemma}

\begin{proof}
For $q=2$ we only have two points so there is nothing to show. For $q=3$, each nonzero distance is equal and we are done. For $q=4$, adjacent vertices have squared distance $2$ and opposite vertices have squared distance $4$. Each nontrivial triangle has side lengths $2,2,4$, so the only tight inequality is $4 = 2+2$ and again we are done.

Now let us consider the case where $q\geq5$ and let us consider three consecutive vertices, we may name them $0,1,2$. This gives $d_q(0,2)=4\sin^2(2\pi/q)$ and $d_q(0,1)+d_q(1,2)=8\sin^2(\pi/q)$.
The triangle inequality, if it was satisfied, would imply that $  4\sin^2(2\pi/q)\le8\sin^2(\pi/q)$. Using $\sin(2\theta)=2\sin\theta\cos\theta$ and cancelling the positive factor $\sin^2(\pi/q)$ gives $\cos^2(\pi/q) \leq \nicefrac{1}{2}$. But note that $q\geq 5$ implies $\nicefrac{\pi}{q} < \nicefrac{\pi}{4}$ i.e.,  $\cos^2(\pi/q)> \nicefrac{1}{2}$ which concludes the proof.
\end{proof}

\begin{remark}
\label{rem:nonobtuse}
For arbitrary points $u,v,w$ in Euclidean space we know that
\[
  \|u-w\|_2^2\le\|u-v\|_2^2+\|v-w\|_2^2
  \quad\Longleftrightarrow\quad
  (u-v)\cdot(w-v)\ge0.
\]
Thus squared Euclidean distance satisfies all triangle inequalities on a finite set exactly when every triangle spanned by the set is nonobtuse. Danzer and Gr\"unbaum proved that  nonobtuse sets in $\R^d$ has at most $2^d$ points \cite{DG62} which in the plane gives us the same upper threshold of four. In the direct trigonometric proof we use above we use the actual regular polygon obstruction that will be used later on.
\end{remark}

\begin{figure}[H]
\centering
\begin{tikzpicture}[scale=0.95, every node/.style={font=\small}]
  \begin{scope}[xshift=0cm]
    \foreach \j in {0,1,2}{\coordinate (t\j) at ({90+120*\j}:1.15);}
    \draw[thick] (t0)--(t1)--(t2)--cycle;
    \foreach \j in {0,1,2}{\fill (t\j) circle (1.7pt);}
    \node at (0,-1.55) {$q=3$: acute};
  \end{scope}
  \begin{scope}[xshift=4cm]
    \foreach \j in {0,1,2,3}{\coordinate (s\j) at ({45+90*\j}:1.15);}
    \draw[thick] (s0)--(s1)--(s2)--(s3)--cycle;
    \draw[dashed] (s0)--(s2);
    \foreach \j in {0,1,2,3}{\fill (s\j) circle (1.7pt);}
    \node at (0,-1.55) {$q=4$: right-angle equality};
  \end{scope}
  \begin{scope}[xshift=8.4cm]
    \foreach \j in {0,1,2,3,4}{\coordinate (p\j) at ({90+72*\j}:1.15);}
    \draw[thick] (p0)--(p1)--(p2)--(p3)--(p4)--cycle;
    \draw[very thick, red!65!black] (p0)--(p1)--(p2);
    \draw[very thick, red!65!black] (p0)--(p2);
    \foreach \j in {0,1,2,3,4}{\fill (p\j) circle (1.7pt);}
    \node at (0,-1.55) {$q=5$: an obtuse triple};
  \end{scope}
\end{tikzpicture}
\caption{The geometric threshold behind exactness. On a regular triangle or square, squared chordal distances satisfy the triangle inequalities. Starting with the pentagon, three consecutive vertices violate the squared-distance triangle inequality.}
\label{fig:polygon-cutoff}
\end{figure}
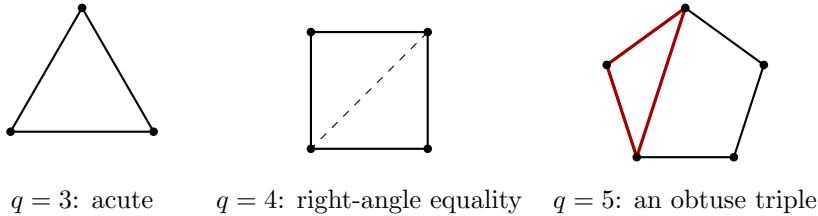

The image of the character $\chi$ is a regular $q(\chi)$-gon. The metric is squared Euclidean by construction and satisfies the triangle inequalities by Lemma \ref{lem:polygon-cutoff}. This gives us the following immediate corollary: 

\begin{corollary}\label{cor:low-q-feasible}
If $q(\chi)\leq 4$, then $d_\chi$ is feasible for GL.
\end{corollary}

The threshold also gives a particularly nice group characterization.

\begin{proposition}
\label{prop:universal-character-feasibility}
Each character metric on $\Gamma$ is feasible for GL if and only if $\exp(\Gamma)\leq 4$.
\end{proposition}

\begin{proof}
If $\exp(\Gamma)\leq 4$ then we have that each character image size divides the exponent and is at most four and so Corollary \ref{cor:low-q-feasible} applies. 
For the opposite direction we note that if $\exp(\Gamma) = m\geq 5$ then the invariant-factor decomposition gives a character with image size $m$. Its regular-polygon metric violates a triangle inequality by Lemma \ref{lem:polygon-cutoff}.
\end{proof}

\subsection{Lossless rounding for images of size two, three, and four}
We now show that for $q \leq 4$ the character metric can also be rounded without loss:

\begin{lemma}
\label{lem:low-order-rounding}
Let $\chi\neq \mathbf{1}$ and let $q=q(\chi)\leq4$. Then $  \roundval(\chi)\le\lambda_\chi(G)$.
In particular (i) if $q=2$ or $q=3$ then $\psi_G(\ker\chi)=\lambda_\chi(G)$ and (ii) if $q=4$ then the best of the two half-circle quotient cuts has sparsity at most $\lambda_\chi(G)$.
\end{lemma}

\begin{proof}
Let $\beta=\Prb_{s\sim S}[\chi(s)\neq 1]$. If $q=2$ the only non-trivial phase is the $-1$. Each step with $\chi(s)\neq 1$ crosses the kernel and every step with $\chi(s)=1$ stays inside a kernel coset. Since $\mu(\ker\chi)=1/2$ we have $N_G(\ker\chi)=\beta$,  and $D(\ker\chi)=\nicefrac{1}{2}$, so $\psi_G(\ker\chi)=2\beta$. Also $\lambda_\chi(G)=1-[(1-\beta)\cdot1+\beta\cdot(-1)]=2\beta$.

If $q=3$, every nontrivial phase has real part $-1/2$. A nonkernel step translates the quotient by a nonzero residue. It crosses the singleton kernel exactly when its starting or ending phase is $0$, two disjoint events of probability $1/3$. Hence $N_G(\ker\chi)=\nicefrac{2}{3} \cdot \beta$,  $  D(\ker\chi)=2\cdot\frac{1}{3}\cdot\frac{2}{3}=\frac{4}{9}$ and $\psi_G(\ker\chi)=\frac{3}{2}\beta$.  The eigenvalue is the same: $\lambda_\chi(G)
  =1-[(1-\beta)\cdot1+\beta\cdot(-1/2)]
  =\nicefrac{3}{2}\beta$ .

Finally suppose $q=4$. We may label the phases $0,1,2,3$ around the square and we let $B_0=\{0,1\},   B_1=\{1,2\}$ and define $A_i=\chi^{-1}(B_i)$, then we get that on the square equal phases are separated by neither cut, adjacent phases are separated by exactly one cut, and finally opposite phases are separated by both which tells us that pointwise on $\Gamma\times\Gamma$ we have that $d_\chi=2d_{A_0}+2d_{A_1}$. Both of these cuts have density $\nicefrac{1}{2}$ and so $D(d_{A_0})=D(d_{A_1})=\nicefrac{1}{2}$. Using the Lemma \ref{lem:character-metric-ratio}
\[
  \lambda_\chi(G)
  ~=~ \frac{N_G(d_\chi)}{D(d_\chi)}
  ~=~ N_G(d_{A_0})+N_G(d_{A_1})
  ~=~ \frac{\psi_G(A_0)+\psi_G(A_1)}{2}
\]
and at least one of the two cuts has sparsity at most $\lambda_\chi(G)$.
\end{proof}

\begin{theorem}\label{thm:low-order-bottom-character}
Let $G=\Cay(\Gamma,S)$. Suppose there exists a nontrivial character $\chi$ such that $\lambda_\chi(G)=\lambda_2(G)$ and $q(\chi)\leq 4$.
Then
\[
  \lambda_2(G) ~=~ \gl(G) ~=~ \psi(G).
\]
Moreover, the optimal cut is explicitly obtained from $\chi$ as in Lemma \ref{lem:low-order-rounding}.
\end{theorem}

\begin{proof}
Using Corollary \ref{cor:low-q-feasible} and Lemma \ref{lem:character-metric-ratio} we have that $d_\chi$ is a feasible metric of objective value $\lambda_\chi(G)=\lambda_2(G)$ i.e., 
$\gl(G)\leq\lambda_2(G)$.  
Together with the spectral lower bound this gives us $\gl(G)=\lambda_2(G)$. By Lemma \ref{lem:low-order-rounding} we get that a quotient cut has sparsity at most $\lambda_\chi(G)=\lambda_2(G)$ and  combining with \eqref{eq:fundamental-chain} forces equality everywhere.
\end{proof}

\begin{corollary}[Exactness for exponent at most four]\label{cor:exp-four-exactness}
If $\exp(\Gamma)\leq 4$, then every Cayley graph $G=\Cay(\Gamma,S)$ satisfies $\lambda_2(G)=\gl(G)=\psi(G)$.
\end{corollary}

\begin{proof}
By equation \eqref{eq:lambda2-character-min} there is some nontrivial character which attains $\lambda_2(G)$ the image size of which divides $\exp(\Gamma)$ and is therefore at most four. Then we just apply Theorem \ref{thm:low-order-bottom-character}.
\end{proof}


We finally note that kernels are not necessarily exact at order four with a simple example. Let $G=C_4=\Cay(\Z_4,\{1,-1\})$ and $\chi(x)=i^x$. The kernel is the singleton $\{0\}$. Four of the eight directed transitions cross it, so $N_G(\{0\})=\nicefrac{1}{2}$. Its denominator is $2 \cdot (1/4) \cdot (3/4)=\nicefrac{3}{8}$ giving $\psi_G(\ker\chi)=\nicefrac{4}{3}$.
But $\lambda_2(C_4)=1$, and the half-circle cut $\{0,1\}$ has numerator and denominator that are both $\nicefrac{1}{2}$ and so the quotient rounding is genuinely stronger than kernel rounding even in the smallest nontrivial case.


\section{Cyclic averaging and approximation}\label{sec:averaging}
When $q(\chi)\geq 5$, the squared chordal distance associated with $\chi$ violates the triangle inequality. This does not imply an integrality gap, since another feasible metric may still attain the cut optimum. In order to treat arbitrary character orders we will keep the GL feasible region unchanged but we will replace $G$ by a cyclically averaged walk in the same spirit as in \cite{dorsiaetall25}. In the next two subsections we will compare the metric objective values on $G$ and the cyclically averaged $G^\sharp$, and then compute the simplified Fourier spectrum of $G^\sharp$. We will then round a minimizing character by a cut in its cyclic quotient.

\subsection{Metric objectives under full cyclic averaging}
Let $r\geq 2$ and $0 \leq \ell \leq r-1$ (note that we include the 0 to be fully cyclic), then define
\begin{equation}\label{eq:Lr-def}
  L_r(\ell)=\min\{\ell,r-\ell\}.
\end{equation}
to be the smallest number of $+/-1$ steps from $0$ to $\ell$ in the cycle $\Z_r$. We also define
\begin{equation}\label{eq:alpha-def}
  \alpha(r)~=~\frac1{r}\sum_{\ell=0}^{r-1}L_r(\ell) 
  ~=~ \begin{cases}
    \nicefrac{r}{4},&r\text{ even},\\[6pt]
    \frac{r^2-1}{4r},&r\text{ odd}
  \end{cases}.
\end{equation}
to be the average distance around a cycle of length $r$ and it measures how much we have to pay if we replace one generator step by a uniformly random move in the cyclic subgroup it generates. For the closed form of $\alpha(r)$ we notice that for $r=2k$ then 
\[
  \sum_{\ell=0}^{r-1}L_r(\ell)
  =2(1+\cdots+(k-1))+k=k^2,
\]
so $\alpha(2k)=k^2/(2k)=k/2=r/4$.
If $r=2k+1$
\[
  \sum_{\ell=0}^{r-1}L_r(\ell)
  =2(1 + \cdots + k)=k(k+1),
\]
so $\alpha(2k+1)=k(k+1)/(2k+1)=\nicefrac{(r^2-1)}{4r}$. 
Moreover, $\alpha(r)$ is nondecreasing and $\alpha(r)=r/4+\mathcal{O}(1)$. 

For the generator multiset $S$ we subsequently define
\begin{equation}\label{eq:rho-def}
  \rho(S)=\max_{s\in S}\alpha(\ord(s))
\end{equation}
which is the worst such cost among all the generators which is precisely the loss we will suffer when we wish to compare GL metrics before and after the cyclic averaging.

We now construct the fully cyclically averaged  graph:
\begin{definition}[Cyclic averaging]\label{def:cyclic-averaging}
The cyclically averaged walk $G^\sharp$ is the symmetric weighted Abelian Cayley random walk generated with the following steps:
\begin{enumerate}
  \item we sample an occurrence $s\sim S$ and $\ell$ uniformly from $\{0,\ldots,r_s-1\}$, then
  \item move from $x$ to $x+\ell s$.
\end{enumerate}
\end{definition}


We note that the inclusion of the zero element turns cyclic averaging into an exact subgroup projection but, on the other hand, the exclusion of it it leaves a shifted  version of that projection that depends on the order.

\subsection{Averaged path inequality}
In order to compare $G^\sharp$ with $G$ we will represent each averaged step $\ell \cdot s$ with a shortest path of $L_{r_s}(\ell)$ steps in the directions $+s$ or $-s$ and we will then average the resulting triangle inequality over the starting vertex.

\begin{lemma}[Averaged cyclic path inequality]\label{lem:averaged-path}
Let $d$ be a metric on $\Gamma$. For each nonzero $s\in\Gamma$ with $r=\ord(s)$ and each $1\leq \ell \leq r-1$ we have that
\begin{equation}\label{eq:averaged-path}
  \E_{x\sim\Gamma}d(x,x+\ell s)
  ~\leq~ L_r(\ell) \cdot \E_{x\sim\Gamma}d(x,x+s).
\end{equation}
\end{lemma}

\begin{proof}
Let $L=L_r(\ell)$. In the cyclic subgroup $\langle s\rangle$ we choose a shortest path $0=y_0,y_1,\ldots,y_L=\ell s$ the increments of which are all either $+s$ or all $-s$. Now for each $x$ the triangle inequality gives us
\[
  d(x,x+\ell s)
  ~\leq~ \sum_{j=0}^{L-1}d(x+y_j,x+y_{j+1}).
\]
We now average over uniform $x$. If $y_{j+1}-y_j=s$, translation invariance gives
\[
  \E_x d(x+y_j,x+y_j+s)=\E_u d(u,u+s).
\]
If the step is $-s$ then symmetry of $d$ and a change of variables give
\[
  \E_xd(x+y_j,x+y_j-s)=\E_ud(u,u-s)=\E_ud(u,u+s).
\]
Each one of the $L$ terms has the same average and this the claim.
\end{proof}

\begin{proposition}\label{prop:metric-averaging}
For each metric $d$ on $\Gamma$,  $N_{G^\sharp}(d)  ~\leq~ \rho(S) \cdot N_G(d)$
and, consequently, $\gl(G^\sharp) \leq \rho(S) \cdot \gl(G)$.
\end{proposition}

\begin{proof}
Let us fix an occurrence $s\in S$. The Averaging Lemma \ref{lem:averaged-path} over uniform $\ell \in\{0,\ldots,r_s-1\}$ gives
\[
  \frac{1}{r_s}\sum_{\ell=0}^{r_s-1}
  \E_xd(x,x+\ell s)
  \leq \alpha(r_s)\E_xd(x,x+s).
\]
We note that the $\ell = 0$ term is zero and we average over $s\sim S$ and use $\alpha(r_s)\le\rho(S)$, the first inequality of the statement.

For a feasible GL metric, the denominator $D(d)=\E_{x,y}d(x,y)$ depends only on the uniform vertex measure, not on the graph. Therefore
\[
  \frac{N_{G^\sharp}(d)}{D(d)}
  ~\leq~ \rho(S) \cdot \frac{N_G(d)}{D(d)}.
\]
Finally, it is enough to take the infimum over the feasible metrics for $G$.
\end{proof}

We remark that  Oveis Gharan  and Trevisan compare a GL solution on $G$ with the same solution on a walk-power graph $G^t$, controlling the increase by decomposing a random long walk and averaging over translations \cite{Trevisan21}. In our Proposition \ref{prop:metric-averaging} above we use the same  averaging idea but on a different graph: \emph{averaging} treats each cyclic generator direction independently which means that the loss depends on the cyclic order and neither on the random-walk length nor on the actual degree.


\subsection{The simplified Fourier spectrum}
The Fourier spectrum of the cyclically averaged walk has the following simplified and useful form:

\begin{lemma}[Spectrum of $G^\sharp$]\label{lem:saturated-spectrum}
For each character $\chi\in\Gammahat$ we have
\begin{equation}\label{eq:saturated-spectrum}
  \lambda_\chi(G^\sharp)
  =\E_{s\sim S} ~\one_{\{\chi(s)\neq 1\}}.
\end{equation}
\end{lemma}

\begin{proof}
We fix an occurrence $s \in S$ and recall that $r_s = \ord(s)$. Then the conditional adjacency contribution is 
\[
\frac{1}{r_s} \sum_{\ell=0}^{r_s-1} \chi(\ell s) 
~=~
\frac{1}{r_s} \sum_{\ell = 0}^{r_s-1} \chi(s)^\ell.
\]
If $\chi(s) = 1$ then this average equals to one so the conditional Laplacian contribution is  zero otherwise the $\chi(s)$ is a non-trivial root of unity the order of which divides the order $r_s$ and the summation above becomes zero.  We average over $s \sim S$ and conclude the statement of the Lemma.

\end{proof}


\subsection{Kernel rounding and the approximation}
The directions under which a character remains unchanged lie in its kernel and they preserve every kernel coset. We can calculate the contribution to the kernel cut of the remaining directions. In the following we assume $\chi\ne\mathbf{1}$, let $q=q(\chi)$.

\begin{lemma}[Kernel rounding with image-size factor]\label{lem:kernel-rounding}
We have that
\begin{equation}\label{eq:kernel-beta}
  \psi_G(\ker \chi)
  ~=~ \frac{q}{q-1} \cdot \Prb_{s\sim S}[\chi(s)\neq 1]
  ~=~ \frac{q}{q-1} \cdot \bigg[\lambda_\chi(G^\sharp)\bigg].
\end{equation}
\end{lemma}

\begin{proof}
We first set $\beta=\Prb_{s\sim S}[\chi(s)\neq 1]$.
By \eqref{eq:kernel-density} the $\mu(K)$ is equal to $1/q$. If $\chi(s)=1$ translation by $s$ preserves every kernel coset and never crosses $K = \ker \chi$. If $\chi(s)\ne1$ then of course $K$ and $K-s$ are disjoint. A directed edge $(x,x+s)$ crosses $K$ when either $x\in K$ or $x\in K-s$ and so the  probability of crossing (over some  $x$) is $2/q$ and so we can conclude that $N_G(K)=\nicefrac{2\beta}{q}$.
The denominator on the other hand is $D(K) = 2 \cdot \nicefrac{1}{q} \cdot \left(1-\nicefrac{1}{q}\right) = \frac{2(q-1)}{q^2}$.
Dividing the two terms gives us \eqref{eq:kernel-beta}. By the previous Lemma \ref{lem:saturated-spectrum} we have that $\lambda_\chi(G^\sharp) = \Pr_s[\chi_s \neq 1]$ so we get the claimed equality.
\end{proof}

The combination of Corollary \ref{cor:quotient-kernel-chain} and Lemma \ref{lem:kernel-rounding} give 
\[
  \psi(G)
  ~\leq~\roundval(\chi^*)
  ~\leq~\psi_G(\ker\chi^*)
  ~=~\frac{q^*}{q^*-1}\cdot \bigg[\lambda_{\chi^*}(G^\sharp)\bigg].
\]
Since the quantity $\chi^*$ minimizes among the nontrivial characters we have that $\lambda_{\chi^*}(G^\sharp)=\lambda_2(G^\sharp)$. 
On the other hand, by the spectral lower bound and the averaging comparison we get
\[
  \lambda_2(G^\sharp)
  ~\leq~\gl(G^\sharp)
  ~\leq~\rho(S) \cdot \gl(G).
\]
Combining the above and using that $q^*/(q^*-1)\leq2$ we get the main approximation theorem: 
\begin{theorem}[Cyclic averaging and quotient rounding]\label{thm:general-bound}
Let $G=\Cay(\Gamma,S)$ and let $\chi^*$ be a nontrivial character minimizing $\lambda_\chi(G^\sharp)$ i.e., $\chi^* \in \arg \min \Pr_s[\chi(s) \neq 1]$. Set $q^*=q(\chi^*)=|\chi^*(\Gamma)|$.
Then
\begin{equation}\label{eq:general-round-bound}
  \psi(G)
  ~\leq~ \roundval(\chi^*)
  ~\leq~ \frac{q^*}{q^*-1}\rho(S)\gl(G).
\end{equation}
The explicit kernel cut $K=\ker\chi^*$ satisfies the same right-hand side in \eqref{eq:general-round-bound}.
\end{theorem}

\begin{corollary}\label{cor:bounded-order}
For connected $G$, if each generator has order at most $R$ then
\[
  \frac{\psi(G)}{\gl(G)}\leq 2 \cdot \alpha(R)
  =
  \begin{cases}
    \nicefrac{R}{2},&R\text{ even}, \\
    \nicefrac{R^2-1}{2R},&R\text{ odd}.
  \end{cases}
\]
\end{corollary}


\subsection{Algorithm}

The proof also yields an explicit algorithm that does not solve the SDP.

\begin{corollary}[Explicit kernel algorithm]\label{cor:kernel-algorithm}
Assume the finite Abelian group is given by an explicit cyclic decomposition and the symmetric generator multiset $S$ is given explicitly. In time polynomial in $n=|\Gamma|$, $|S|$, and the bit length of the decomposition, one can output a cut $K$ satisfying
\[
  \psi_G(K)
  ~\leq~ \frac{q^*}{q^*-1} \cdot \rho(S) \cdot \gl(G).
\]
\end{corollary}

\begin{proof}
The explicit decomposition lists all $n$ characters as product characters. For each nontrivial $\chi$ we compute \eqref{eq:saturated-spectrum} exactly. This is easy as it only requires for testing whether $\chi(s)=1$ and for knowing the $\ord(s)$. Choose a minimizer $\chi^*$ and output its kernel. The value guarantee follows by Theorem \ref{thm:general-bound} and by using $\gl(G) \leq \psi(G)$.
\end{proof}

\begin{remark}
Enumerating all $|\Gamma|$ characters is polynomial in $n=|\Gamma|$. It can be exponential in a succinct description whose bit length is only $\mathcal{O}(\log|\Gamma|)$. The algorithmic claim is made in the explicit $n$-vertex/group decomposition model.
\end{remark}


Finally, we parallel Lemma 8.2 from d'Orsi et al. \cite{dorsiaetall25} for Cayley graphs over $\mathbb{F}_p^k$, $p$ prime.

\begin{corollary}\label{cor:prime-exponent}
Let $p$ be prime and let $G=\Cay(\mathbb{F}_p^k,S)$ be connected. If $p\ge5$, then
\begin{equation}\label{eq:prime-bound}
  \frac{\psi(G)}{\gl(G)}
  ~\leq~\frac{p}{p-1} \cdot \alpha(p) 
  = \frac{p}{p-1} \cdot \frac{(p-1)(p+1)}{4p} 
  ~=~ \frac{p+1}{4}.
\end{equation}
\end{corollary}

\begin{proof}
Every nonzero generator has order $p$, so $\rho(S)=\alpha(p)$. Every nontrivial character of $\mathbb{F}_p^k$ has image size exactly $p$, so $q^*=p$. We apply Theorem \ref{thm:general-bound} and use $\frac{p}{p-1} \alpha(p)=\frac{p+1}{4}$ for odd $p$.
\end{proof}

\begin{remark}
In this prime field setting, our metric comparison is an SDP extension of d’Orsi et al. Lemma 8.2  for the comparison for cut metrics. They prove a  $\nicefrac{p+1}{4}$ bound but in our case this holds more generally \emph{for each Goemans–Linial feasible metric} instead of only cut metrics as in d'Orsi et al. When we combine our bound with the exact kernel identity  we are able to prove the integrality gap bound $\psi(G)/\gl(G) \leq \nicefrac{p+1}{4}$ (something that does not, at least not directly, follow from d'Orsi et al.). 
\end{remark}



\subsection{Exact averaging constant}
The two constants in Theorem \ref{thm:general-bound} arise from different steps. We show that the metric averaging factor $\alpha(r)$ is exact.

\begin{theorem}\label{thm:alpha-optimal}
For every $r\ge2$,
\begin{equation}\label{eq:alpha-optimal}
  \sup_d
  \frac{
    \frac1{r}\sum_{\ell=0}^{r-1}\E_{x\in\Z_r}d(x,x+\ell)
  }{
    \E_{x\in\Z_r}d(x,x+1)
  }
  =\alpha(r),
\end{equation}
where the supremum is over all metrics $d$ with positive denominator. The supremum is attained by a cut metric and hence also within the GL-feasible class.
\end{theorem}

\begin{proof}
The upper bound is Lemma \ref{lem:averaged-path} averaged over $\ell$. For the lower bound, let $m=\lfloor r/2\rfloor$ and take the interval $A=\{0,1,\ldots,m-1\}\subset\Z_r$. We use its cut metric $d_A$. For a shift $\ell$ the symmetric difference of the interval and its translation has size $|A\triangle(A-\ell)|~=~2L_r(\ell)$.

This is easy to see once we choose the shorter orientation around the cycle. If we shift an interval by $L_r(\ell)$ we remove exactly $L_r(\ell)$ points at one boundary but at the same time we add the same number at the other. In the even antipodal case the two intervals are complementary and the same formula gives $r$. Therefore
\[
  \E_xd_A(x,x+\ell)=\frac{2L_r(\ell)}r,
  \qquad
  \E_xd_A(x,x+1)=\frac2r.
\]
The ratio in \eqref{eq:alpha-optimal} is exactly $  \frac1{r}\sum_{\ell=0}^{r-1}L_r(\ell)=\alpha(r)$.
\end{proof}

\section{Some examples and some consequences}\label{sec:examples}

We  give some examples that illustrate both the strength and the limitations of the  bounds of the previous sections. In particular, we discuss how $\rho(S)$ may be far away from the actual Goemans–Linial gap. 

\subsection{Cubelike graphs and exponent three and four}
A cubelike graph is a Cayley graph over $\mathbb{F}_2^k$. By Corollary \ref{cor:exp-four-exactness} each cubelike graph satisfies $\lambda_2(G)= \gl(G)= \psi(G)$ and an optimal cut is a hyperplane, namely the kernel of a bottom character.

We note that the  exact spectral solvability is already observed by d’Orsi et al. \cite{dorsiaetall25}: the Boolean Fourier characters are $\{\nicefrac{+1}{-1}\}$ valued and a bottom character indicates an optimal cut but our results recover this and extends the same exactness to all finite Abelian groups of exponent three and four.


\subsection{Cycles}

The  bound in Theorem \ref{thm:general-bound} is weak on a cycle because the generator has order $n$. On the positive side, the underlying averaging inequality is strong enough to determine the exact GL value.

\begin{proposition}
\label{prop:cycle-exact}
For each $n\geq 3$,
\begin{equation}\label{eq:cycle-exact}
  \gl(C_n)=\psi(C_n)
  =
  \begin{cases}
    4/n,& n \text{ is even},\\[6pt]
    (4n)/(n^2-1),&n\text{ is odd}.
  \end{cases}
\end{equation}
In fact, the same value is obtained if the SDP is relaxed further to all metrics satisfying the triangle inequalities, without the squared-Euclidean requirement.
\end{proposition}

\begin{proof}
Let $d$ be any metric on $\Z_n$. Averaging the generator $1$ makes the random step uniform over $\Z_n$ so $C_n^\sharp=\Cay(\Z_n,\Z_n)$ in normalized random-walk form where the zero step is just a self loop. Since $d(x,x)=0$ we have that
\[
  N_{C_n^\sharp}(d)
  =\frac1{n^2}\sum_{x \neq y}d(x,y) ~=~ D(d).
\]
By Proposition \ref{prop:metric-averaging} we have $D(d) \leq \alpha(n) \cdot N_{C_n}(d)$ and so every metric satisfies
\begin{equation}\label{eq:cycle-metric-lower}
  \frac{N_{C_n}(d)}{D(d)}
  ~\geq~\frac{1}{\alpha(n)}.
\end{equation}

Now let $m=\lfloor n/2\rfloor$ and take the interval cut $A=\{0,1,\ldots,m-1\}$. Exactly two undirected cycle edges cross the cut, equivalently four of the $2n$ directed transitions cross, so $N_{C_n}(A)=\nicefrac{2}{n}$. Its denominator is $D(A) = 2 \cdot \frac{m}{n} \left(1-\frac{m}{n}\right) = \frac{2m(n-m)}{n^2}$.
Therefore
\[
  \psi_{C_n}(A)=\frac{n}{m(n-m)}
  =
  \begin{cases}
    4/n,&n\text{ even},\\[3pt]
    4n/(n^2-1),&n\text{ odd}.
  \end{cases}
\]
Using the closed form for $\alpha$ the lower bound in \eqref{eq:cycle-metric-lower} is exactly the same quantity. Hence the metric relaxation, the GL relaxation, and the cut optimum all coincide.
\end{proof}

\begin{remark}
\label{remark:gaps}
For cycles $\rho(\{\nicefrac{-1}{+1}\})=\alpha(n)=\Theta(n)$ and $\frac{\psi(C_n)}{\gl(C_n)}=1$. Thus $\rho(S)$ can overestimate the true GL gap by a linear factor. At the same time $\lambda_2(C_n)=1-\cos(2\pi/n)=\Theta(n^{-2})$ but $\gl(C_n)=\Theta(n^{-1})$ from which we get that
\[
  \frac{\gl(C_n)}{\lambda_2(C_n)}=\Theta(n).
\]
\end{remark}

\subsection{Discrete tori}\label{subsec:tori}

Fix $r\ge5$ and $k\ge1$, and consider the Cartesian power of a cycle $C_r^{\square k}
  =\Cay((\Z_r)^k,\{\pm e_1,\ldots,\pm e_k\})$ which we denote by $T_{r,k}$. 
This graph has degree $2k$ and each generator has order $r$. 

\begin{proposition}
\label{prop:torus}
Let $a=\lfloor r/2\rfloor$. Then
\begin{equation}\label{eq:torus-lambda}
  \lambda_2(T_{r,k}) ~=~ \frac{1-\cos(2\pi/r)}{k}
\end{equation}
and
\begin{equation}\label{eq:torus-cut}
  \psi(T_{r,k}) ~\leq~ \frac{r}{k \cdot a(r-a)}.
\end{equation}
As a result
\begin{equation}\label{eq:torus-gap}
  \frac{\psi(T_{r,k})}{\gl(T_{r,k})}
  ~\leq~
  \frac{r}{a(r-a)(1-\cos(2\pi/r))},
\end{equation}
a constant depending on $r$ but not on $k$. 
\end{proposition}

\begin{proof}
For $t=(t_1,\ldots,t_k)\in(\Z_r)^k$, the character $\chi_t$ has Laplacian eigenvalue equal to $\lambda_t = 1-\frac{1}{k}\sum_{j=1}^k\cos(2\pi t_j/r)$.
For a nonzero $t$, at least one coordinate is nonzero, and the largest cosine at a nonzero residue is $\cos(2\pi/r)$. We get the minimum eigenvalue by taking one coordinate equal to plus/minus $1$ and all others zero, giving \eqref{eq:torus-lambda}.

Take the coordinate interval cut $A=\{x:x_1\in\{0,1,\ldots,a-1\}\}$. 
Only the two directions $+/- e_1$ can cross. Conditioning on choosing one of those directions, exactly two of the $r$ possible values of $x_1$ can cross the interval boundary so $N_{T_{r,k}}(A)=\frac1k\cdot\frac2r=\frac{2}{kr}$. The density is $a/r$ and so $D(A)=2\frac{a}{r} \left(1-\frac{a}{r} \right)=\frac{2a(r-a)}{r^2}$. Their ratio is \eqref{eq:torus-cut}. Finally, $\gl \geq  \lambda_2$ gives \eqref{eq:torus-gap}.

For $r=5$, $a=2$ and $1-\cos(2\pi/5)=(5-\sqrt5)/4$ which gives $\nicefrac{5}{6(1-\cos(2\pi/5))}=\nicefrac{(5+\sqrt5)}{6}$.
\end{proof}


\subsection{An  $\nicefrac{16}{15}$ integrality gap family}\label{subsec:integrality}
We next give an explicit connected Abelian Cayley graph on which the Goemans–Linial relaxation is not exact, and then extend it to an infinite family. To the best of our knowledge, no previously published explicit finite Abelian Cayley graph with uniform demands has been shown to have Goemans–Linial gap strictly larger than one. The seed graph has ten vertices and gap $16/15$.

We begin with the initial seed graph. Let $C=\operatorname{Cay}\bigl(\mathbb Z_{10},\{\pm1,\pm3\}\bigr)$. Combinatorially, $C$ is $K_{5,5}$ with a perfect matching removed. Then we will show that 
$\lambda_2(C)=\operatorname{SDP_{GL}}(C)=\nicefrac{3}{4}$ and $\psi(C)=\nicefrac{4}{5}$ and so for this small graph we have 
\[
\frac{\psi(C)}{\operatorname{SDP_{GL}}(C)}=\frac{16}{15}.
\]

Indeed for $j\in\mathbb Z_{10}$  let $\chi_j(t)=e^{2\pi ijt/10}$. Proposition \ref{prop:character-diagonalization} about character diagonalization gives $1-\frac12\left( \cos\frac{\pi j}{5} + \cos\frac{3\pi j}{5} \right)$. The smallest nonzero eigenvalue is attained at $j\in \{1,3,7,9\}$, and equals to $\lambda_2(C)=\nicefrac{3}{4}$.

To obtain a matching GL solution we define the translation-invariant semimetric $d_0 = \nicefrac{1}{4}d_{\chi_1} + \nicefrac{1}{4}d_{\chi_3}$ given by the formula
\[
d_0(x,y) ~=~ \frac{1}{2}\left(1-\cos\frac{\pi(y-x)}5\right)
+
\frac{1}{2}\left(1-\cos\frac{3\pi(y-x)}5\right).
\]
This is squared Euclidean distance because it is a nonnegative combination of squared character distances. Its values depend only on $t=y-x$ and are 
\[
d_0(t)=
\begin{cases}
0, & t=0 \\
\nicefrac{3}{4}, & t\in \{ -3,-1,+1,+3 \} \\
\nicefrac{5}{4}, & t\in \{-4,-2,+2,+4\} \\
2, & t=5.
\end{cases}
\]
We first verify the triangle inequalities. If the longest side has length at most $5/4$ then (unless one of the other sides is zero) their sum is at least $\frac34+\frac34>\frac54$. We now  consider a triangle the longest side of which has length $2$, which corresponds to two points whose difference is $5$. In any decomposition $5=u+v$ with $u,v\neq 0$ one term of the sum must be  odd while the other must be even number and none can be $0$ or $5$ and therefore $d_0(u)+d_0(v) = \frac34+\frac54=2 = d_0(5)$. Thus $d_0$ is feasible for the GL relaxation.


Now, each generator belongs to $S = \{ +/-1, +/-3 \}$ and so $\mathbb E_{s\in S}d_0(s) = \nicefrac{3}{4}$ which, together with the observation that $D(d_0) =1$, gives that $\operatorname{SDP_{GL}}(C)\leq \nicefrac{3}{4}$. Together with the general spectral lower bound $\lambda_2(C)\leq \gl(C)$ proves to us that $\operatorname{SDP_{GL}}(C)=\lambda_2(C)=\nicefrac{3}{4}$.

It remains to calculate the  optimum cut. Let $A \subseteq V(C)$ and let $a= |A| \leq 5$ and denote by $c(A)$ to be the associated edge boundary. Since $C$ is four-regular and bipartite $c(A)=4a - 2e(A) \geq 4a - 2 \lfloor \frac{a^2}{4} \rfloor$ which gives $\psi(A)_C = \nicefrac{5c(A)}{2a(10 - a)}$. We get the minimum for $a=5$ for which we get $\psi(C) = \nicefrac{4}{5}$ i.e., 
\[
\lambda_2(C) = \gl(C) = \frac{3}{4} < \frac{4}{5} = \psi(C)
\]
giving a gap of $\nicefrac{16}{15}$. Note that equality is achieved by $A = \{0,1,2,3,4\}$.

The interesting thing is that this exact construction is also present under Cartesian products. We note that the Cartesian product of two Cayley graphs is also Cayley. 

Indeed let us define the $k$ factor Cartesian product of $C$ with itself as $G_k$ i.e., $G_k = C^{\square k} $. We amplify this seed by Cartesian powers, following the product-gap framework of Bonsma \cite{bonsmatensor} and Sachdeva and Tulsiani \cite{DBLP:journals/corr/abs-1105-3383}. The normalized product Laplacian is just the average of coordinate Laplacians giving $\lambda_2(G_k) = \nicefrac{\lambda_2(C)}{k} = \nicefrac{3}{4k}$. The metric defined by
\[
d_k (x,y) = \frac{1}{k} \sum_{i=1}^k d_0(x_i - y_i)
\]
is feasible for GL  with objective function value equal to $\nicefrac{3}{4k}$ i.e., $\gl(G_k) = \nicefrac{3}{4k}$. Notice that as before we have $D(d_k)=1 $ and $N_{G^k} (d_k) = \nicefrac{3}{4k}$.
The required tensorization of uniform sparsest cut is a normalized special case of a more general Cartesian-product result of Bonsma's Theorem 5 \cite{bonsmatensor}, see also \cite{DBLP:journals/corr/abs-1105-3383} Theorem A.1. 

\begin{lemma}
\label{lem:cartesian-tensorization}
Let $H$ be a finite regular multigraph, and equip its $k$-fold
Cartesian power $H^{\square k}$ with the normalized random walk that
first chooses a coordinate $i\in[k]$ uniformly and then performs one
normalized $H$-step in that coordinate. Then $\psi(H^{\square k})=\frac{\psi(H)}{k}$.
\end{lemma}

The product identity itself is therefore known but what matters here to us is that the seed graph above has a nontrivial GL gap and that the tensorized feasible GL metric scales by the same $1/k$ factor.
By Lemma~\ref{lem:cartesian-tensorization}, the optimum cut value
tensorizes with exactly the same normalization: $\psi(G_k) = \frac{\psi(C)}{k} = \frac{4}{5k}$ giving $\frac{\psi(G_k)}{\operatorname{SDP}_{GL}(G_k)} = \frac{16}{15}$
Thus the GL relaxation is not exact on Abelian Cayley graphs, even
for a connected family of growing order. In this family the SDP value
coincides with the spectral lower bound, while every cut is separated
from it by the constant factor $16/15$.

Combining Theorem \ref{thm:general-bound} with the Trevisan-Oveis Gharan estimate gives the following statement for connected $G$
\begin{equation}\label{eq:combined-picture}
  \frac{\psi(G)}{\gl(G)}
  ~\leq~ \min \Big\{2\rho(S), \mathcal{O} (\sqrt{|S|})\Big\}.
\end{equation}
The first term  is effective for higher degree graphs generated by smaller order elements and vice versa for the second term. Cycles and prime-exponent vector spaces are the two extremals of this comparison.

\section{Discussion and open problems}\label{sec:discussion}
Trevisan asked whether Abelian Cayley graphs might admit a constant factor approximation ratio, and specifically suggested that the GL relaxation itself might have a constant gap \cite{Trevisan21}. We provide some evidence but we do not resolve the question:
\begin{enumerate}
  \item a lower order character produces a valid GL metric and we can  round it without loss in its cyclic quotient, and
  \item an arbitrary character simplifies  after cyclic averaging and we can  round it to its kernel.
\end{enumerate}
Our framework is exact in several families although it does not determine the worst possible GL gap on Abelian Cayley graphs.


\begin{question}
Can $\psi(G)/\gl(G)$ grow as a function of $\rho(S)$ on connected Abelian Cayley graphs i.e., is the bound $2\rho(S)$ ever close to the true gap, or is it inherently loose because averaging, the spectral lower bound, and kernel rounding cannot be tight simultaneously?
\end{question}

\begin{question}
Is it possible to analyze $\roundval(\chi^*)$ directly instead of upper-bounding it by the kernel cut in order to obtain a general factor smaller than $q^*/(q^*-1)$ or a bound that depends on the quotient generator distribution?
\end{question}


\section{Usage of GenAI in the manuscript}
The project was conceived and developed by the author with the following exceptions for which, for the sake of transparency, the author acknowledges the usage of GenAI (ChatGPT \emph{Plus} 5.5 \& 5.6): firstly, GenAI was used for some editorial assistance throughout the manuscript and was used more extensively to help organize, reformulate, and present the mathematical arguments and ideas already developed by the author in the final write-up of Subsections 5.1 and 5.2. More substantial mathematical assistance was received in the following: Remark \ref{remark:gaps} was formulated and organized by GenAI.  Secondly,  and regarding the integrality gap example in Subsection \ref{subsec:integrality}, although the seed graph and its analysis are the author’s, GenAI  suggested using Cartesian products to extend the example to an infinite family and identified Bonsma's work \cite{\cite{bonsmatensor}} as relevant to this extension . Finally, this directly inspired also the discussion of Tori in subsection \ref{subsec:tori}. The author independently checked all statements, calculations, and references where GenAI was used and takes full responsibility for its contents.

\bibliographystyle{plain}
\bibliography{cayley_arv} 

@article{ARV09,
  author  = {Sanjeev Arora and Satish Rao and Umesh V. Vazirani},
  title   = {Expander Flows, Geometric Embeddings and Graph Partitioning},
  journal = {Journal of the ACM},
  volume  = {56},
  number  = {2},
  pages   = {5:1--5:37},
  year    = {2009},
  doi     = {10.1145/1502793.1502794}
}

@inproceedings{CNR25,
  author    = {Alan Chang and Assaf Naor and Kevin Ren},
  title     = {Optimal Rounding for Sparsest Cut},
  booktitle = {Proceedings of the 57th Annual ACM Symposium on Theory of Computing (STOC 2025)},
  pages     = {643--652},
  publisher = {Association for Computing Machinery},
  year      = {2025},
  doi       = {10.1145/3717823.3718285}
}

@inproceedings{DBLP:conf/stoc/KwokLLGT13,
  author       = {Tsz Chiu Kwok and
                  Lap Chi Lau and
                  Yin Tat Lee and
                  Shayan Oveis Gharan and
                  Luca Trevisan},
  editor       = {Dan Boneh and
                  Tim Roughgarden and
                  Joan Feigenbaum},
  title        = {Improved Cheeger's inequality: analysis of spectral partitioning algorithms
                  through higher order spectral gap},
  booktitle    = {Symposium on Theory of Computing Conference, STOC'13, Palo Alto, CA,
                  USA, June 1-4, 2013},
  pages        = {11--20},
  publisher    = {{ACM}},
  year         = {2013},
  url          = {https://doi.org/10.1145/2488608.2488611},
  doi          = {10.1145/2488608.2488611},
  bibsource    = {dblp computer science bibliography, https://dblp.org}
}

@article{DG62,
  author  = {Ludwig Danzer and Branko Gr{\"u}nbaum},
  title   = {{\"U}ber zwei Probleme bez{\"u}glich konvexer {K}{\"o}rper von {P.} {E}rd{\H{o}}s und von {V.} {L.} {K}lee},
  journal = {Mathematische Zeitschrift},
  volume  = {79},
  pages   = {95--99},
  year    = {1962},
  doi     = {10.1007/BF01193107}
}

@book{DL97,
  author    = {Michel M. Deza and Monique Laurent},
  title     = {Geometry of Cuts and Metrics},
  series    = {Algorithms and Combinatorics},
  volume    = {15},
  publisher = {Springer},
  address   = {Berlin},
  year      = {1997},
  doi       = {10.1007/978-3-642-04295-9},
  isbn      = {978-3-540-61611-5}
}

@article{bonsmatensor,
author = {Paul Bonsma},
title = {Sparsest cuts and concurrent flows in product graphs},
journal = {Discrete Applied Mathematics},
volume = {136},
number = {2-3},
pages = {173-182},
year = {2004},
issn = {0166-218X},
doi = {https://doi.org/10.1016/S0166-218X(03)00439-6},
url = {https://www.sciencedirect.com/science/article/pii/S0166218X03004396},
author = {Paul Bonsma}
}

@inproceedings{dorsiaetall25,
  author    = {Tommaso d'Orsi and Chris Jones and Jake Ruotolo and Salil Vadhan and Jiyu Zhang},
  title     = {Sparsest Cut and Eigenvalue Multiplicities on Low Degree Abelian Cayley Graphs},
  booktitle = {Approximation, Randomization, and Combinatorial Optimization. Algorithms and Techniques (APPROX/RANDOM 2025)},
  series    = {Leibniz International Proceedings in Informatics},
  volume    = {353},
  pages     = {16:1--16:20},
  publisher = {Schloss Dagstuhl--Leibniz-Zentrum f{\"u}r Informatik},
  year      = {2025},
  doi       = {10.4230/LIPIcs.APPROX/RANDOM.2025.16},
  note      = {Full version: arXiv:2412.17115}
}

@inproceedings{KM13,
  author    = {Daniel M. Kane and Raghu Meka},
  title     = {A {PRG} for Lipschitz Functions of Polynomials with Applications to Sparsest Cut},
  booktitle = {Proceedings of the 45th Annual ACM Symposium on Theory of Computing (STOC 2013)},
  pages     = {1--10},
  publisher = {Association for Computing Machinery},
  year      = {2013},
  doi       = {10.1145/2488608.2488610}
}

@article{LR99,
  author  = {Tom Leighton and Satish Rao},
  title   = {Multicommodity Max-Flow Min-Cut Theorems and Their Use in Designing Approximation Algorithms},
  journal = {Journal of the ACM},
  volume  = {46},
  number  = {6},
  pages   = {787--832},
  year    = {1999},
  doi     = {10.1145/331524.331526}
}

@article{LLR95,
  author  = {Nathan Linial and Eran London and Yuri Rabinovich},
  title   = {The Geometry of Graphs and Some of Its Algorithmic Applications},
  journal = {Combinatorica},
  volume  = {15},
  number  = {2},
  pages   = {215--245},
  year    = {1995},
  doi     = {10.1007/BF01200757}
}

@inproceedings{NY17,
  author    = {Assaf Naor and Robert Young},
  title     = {The Integrality Gap of the {Goemans--Linial} {SDP} Relaxation for Sparsest Cut is at Least a Constant Multiple of {$\sqrt{\log n}$}},
  booktitle = {Proceedings of the 49th Annual ACM Symposium on Theory of Computing (STOC 2017)},
  pages     = {564--575},
  publisher = {Association for Computing Machinery},
  year      = {2017},
  doi       = {10.1145/3055399.3055413}
}

@article{Schoenberg38,
  author  = {I. J. Schoenberg},
  title   = {Metric Spaces and Positive Definite Functions},
  journal = {Transactions of the American Mathematical Society},
  volume  = {44},
  number  = {3},
  pages   = {522--536},
  year    = {1938},
  doi     = {10.1090/S0002-9947-1938-1501980-0}
}

@book{Terras99,
  author    = {Audrey Terras},
  title     = {Fourier Analysis on Finite Groups and Applications},
  series    = {London Mathematical Society Student Texts},
  volume    = {43},
  publisher = {Cambridge University Press},
  address   = {Cambridge},
  year      = {1999},
  isbn      = {978-0-521-45718-7}
}

@misc{Trevisan21,
  author       = {Shayan Oveis Gharan and Luca Trevisan},
  title        = {{ARV} on Abelian Cayley Graphs},
  howpublished = {\emph{In theory} blog},
  date         = {08 October 2021},
  year         = {2021},
  note          = {https://lucatrevisan.wordpress.com/2021/10/08/arv-on-abelian-cayley-graphs/}
}

@book{SteinbergGroups,
 author = {Benjamin Steinberg}, 
 title = {Representation Theory of Finite Groups - An Introductory Approach}, 
 publisher = {Springer New York, NY},
 series = {Universitext},
 date =  {23 October 2011},
 year = {2011},
 isbn = {978-1-4614-0776-8}
 }

@article{DBLP:journals/dcg/AustinNV10,
  author       = {Tim Austin and
                  Assaf Naor and
                  Alain Valette},
  title        = {The Euclidean Distortion of the Lamplighter Group},
  journal      = {Discret. Comput. Geom.},
  volume       = {44},
  number       = {1},
  pages        = {55--74},
  year         = {2010},
  url          = {https://doi.org/10.1007/s00454-009-9162-6},
  doi          = {10.1007/S00454-009-9162-6},
  bibsource    = {dblp computer science bibliography, https://dblp.org}
}

@article{DBLP:journals/toc/NewmanR09,
  author       = {Ilan Newman and
                  Yuri Rabinovich},
  title        = {Hard Metrics from Cayley Graphs of Abelian Groups},
  journal      = {Theory Comput.},
  volume       = {5},
  number       = {1},
  pages        = {125--134},
  year         = {2009},
  url          = {https://doi.org/10.4086/toc.2009.v005a006},
  doi          = {10.4086/TOC.2009.V005A006},
  bibsource    = {dblp computer science bibliography, https://dblp.org}
}

@article{DBLP:journals/corr/abs-1105-3383,
  author       = {Sushant Sachdeva and
                  Madhur Tulsiani},
  title        = {Cuts in Cartesian Products of Graphs},
  journal      = {CoRR},
  volume       = {abs/1105.3383},
  year         = {2011},
  url          = {http://arxiv.org/abs/1105.3383},
  eprinttype   = {arXiv},
  eprint       = {1105.3383},
  bibsource    = {dblp computer science bibliography, https://dblp.org}
}

@book{kantor2015mathematics,
  title     = {Mathematics++: Selected Topics Beyond the Basic Courses},
  author    = {Kantor, Ida and Matou{\v{s}}ek, Ji{\v{r}}{\'\i} and {\v{S}}{\'a}mal, Robert},
  year      = {2015},
  publisher = {American Mathematical Society},
  series    = {Student Mathematical Library},
  volume    = {75},
  address   = {Providence, Rhode Island},
  isbn      = {978-1-4704-2261-5}
}

@book{dummitfoote,
author = {Dummit, David S. and Foote, Richard M.},
year = {2004},
title = {Abstract algebra (3rd ed.)}, 
publisher = {Wiley, New York},
isbn = {ISBN 978-0-471-43334-7}
}

@book{analysistopologyhowes,
title ={Modern Analysis and Topology},
author = {Norman R. Howes},
series = {Universitext},
publisher = {Springer, New York},
year = {1995},
isbn = {978-0-387-97986-1}
}
\end{document}